\documentclass[12pt,reqno]{article}

\usepackage[usenames]{color}
\usepackage{amssymb}
\usepackage{graphicx}
\usepackage{amscd}

\usepackage[colorlinks=true,
linkcolor=webgreen,
filecolor=webbrown,
citecolor=webgreen]{hyperref}

\definecolor{webgreen}{rgb}{0,.5,0}
\definecolor{webbrown}{rgb}{.6,0,0}

\usepackage{color}
\usepackage{fullpage}
\usepackage{float}

\usepackage{graphics,amsmath,amssymb}
\usepackage{amsthm}
\usepackage{amsfonts}
\usepackage{latexsym}
\usepackage{epsf}

\usepackage{verbatim}

\usepackage{tikz}
\usepackage{forest}

\usepackage{mathtools}

\theoremstyle{plain}
\newtheorem{theorem}{Theorem}

\newtheorem{proposition}[theorem]{Proposition}

\theoremstyle{definition}

\theoremstyle{remark}

\begin{document}

\title{From Unequal Chance to a Coin Game Dance: Variants of Penney's Game}
\author{
Isha Agarwal \and
Matvey Borodin \and
Aidan Duncan \and
Kaylee Ji \and
Shane Lee \and
Boyan Litchev \and
Anshul Rastogi \and
Garima Rastogi \and
Andrew Zhao \\
PRIMES STEP\\
\\
Tanya Khovanova\\
MIT
}

\maketitle

\begin{abstract}
We introduce and analyze several variations of Penney's game aimed to find a more equitable game.
\end{abstract}

\section{Penney's game}

Alice and Bob, two firm and fair friends with spare change and spare time, have a fair coin to flip. Alice selects a string of flip outcomes (heads or tails) of length $n$, after which Bob chooses his very own string of flip outcomes, also of length $n$. They then begin by tossing the coin. Whoever's string appears first in the sequence of heads and tails that they flip is the victor.

This game that Alice and Bob have ventured out to play is known as ``Penney's game'' or ``paradoxical pennies.'' We denote flip outcomes with H for heads and T for tails.

Suppose Alice and Bob play the game for $n=3$, where Alice selects the string HHH and Bob selects the string THH. They then toss the coin several times and yield the following sequence of flips:
\[\text{H H T T H }\textbf{T H H}.\]
We see that Bob's selected string, emphasized in the sequence of flip outcomes above, appears before Alice's chosen string HHH does. As a result, he is the winner of this round.
	
As Alice and Bob soon discover, Penney's game is in fact far more intriguing than they ever anticipated. This is due to how deceptively counterintuitive the nature of the game truly is. The dazzlingly daring duo decide to play several rounds with these same strings. They find that Bob accumulates a rather high amount of wins in comparison to Alice.

Suspecting something, they decide to peer at Penney's game a little closer. Alice correctly commences with the claim that each string of length 3 comprised of Hs and Ts has the same probability of $\frac{1}{2}\times\frac{1}{2}\times\frac{1}{2}=\frac{1}{8}$ to appear in three consecutive tosses. Alice and Bob then incorrectly assume that the game must be fair as a result. However, this assumption contradicts what was occurring --- even with a fair coin, Bob was triumphant in the great majority of rounds.

Befuddled, Alice and Bob approach us for aid as they lament about their predicament, a hopeful gleam scintillating in their eyes. We propose to them that perhaps they should approach their fairness assumption by considering which one of HH or HT is expected to show up in fewer flips.

Concurring with a nod, Alice correctly starts, ``In both cases, we need to wait for the first H. After that, an unfavorable flip for the former string, HH, means restarting entirely. However, while an unfavorable flip for the latter string, HT, means it fails to progress, it doesn't need to restart!''
	
``So, the latter string, HT, has a shorter expected wait time!'' Bob chimes in a conclusion. ``This certainly shows that our initial assumption that the game was fair is not substantiated. Something is indubitably afoot! Perhaps it has to do with expected wait times.''
	
``Aye,'' Alice amiably agrees. ``Is there a way to find the exact expected wait time in terms of flips?''
	
We tell them that, as a matter of fact, there is. We cover this further along in the paper, in Section~\ref{sec:sometheory}. For the time being, we provide them with the following data for reference: the expected wait time for HT is 4, while the wait time for HH is 6. For the string of length 3, the expected wait time for HHH and TTT is 16 flips, the expected wait time for HTH and THT is 14 flips, and the expected wait time for the four remaining strings --- HHT, HTT, THH, and TTH --- is 8 flips.

Now illuminated by this enrapturing revelation, the two companions present to us a second erroneous argument: ``If a string has a shorter wait time,'' Bob incorrectly states, ``it must bear a greater probability of winning since it is more likely to appear first. If two strings have the same wait time, then the probabilities of winning are the same.''
	
However, we inform them that this argument is incorrect as well. We explain to them, ``Suppose that you, Alice, select the string HHH and you, Bob, choose HHT as your string. As we previously mentioned, Alice's string has a much longer expected wait time of 16 flips while Bob's has a shorter expected wait time of 8 flips.''

``Right,'' Alice and Bob eagerly nod along.
	
``Let's look at how these strings would play out in the actual game. Both of you wait for two Hs in a row since that segment of your selected strings is the same. After two Hs in a row appear, however, each of you wins with probability $\frac{1}{2}$.''
	
``I see,'' Alice replies. ``So, as you have elucidated, despite our selection here having drastically different wait times, our probabilities of winning are, in fact, equal.''
	
``Correct,'' we reply. ``Yet another example where the string with a longer expected wait time has an advantage is where you, Alice, select the string THTH, which has an expected wait time of 20 flips, and play against Bob's choice of HTHH, which has an expected wait time of 18 flips. This case lies in your favor as your probability of winning is $\frac{9}{14}$, although your string has the higher expected wait time.''

The formula for winning probabilities in Penney's game is well-known. We present it in Section~\ref{sec:sometheory}. Table~\ref{table:PABobProb} displays the probability of Bob winning in the game with strings of length 3. Alice's possible choices fill the top row and Bob's possible choices are in the left column. We only show Bob's winning probabilities for Alice's potential strings that begin with H, as we can calculate the rest by symmetry.

\begin{table}[h!]
\begin{center}
\begin{tabular}{ |c|c|c|c|c| } 
  \hline
     & HHH & HHT & HTH & HTT \\
  \hline
 HHH& 	*  & 	1/2  &	2/5   & 2/5 \\
  \hline
 HHT& 	1/2  & 	*  & 	2/3  & 	2/3	  \\
  \hline
 HTH& 	3/5 & 	1/3  &  * &  	1/2 \\
  \hline
 HTT& 	3/5   &	1/3   & 1/2  & 	*  \\
  \hline
 THH& 	7/8 &  	3/4 &	1/2   &	1/2  \\
  \hline
 THT& 	7/12  & 3/8  & 	1/2  & 	1/2 \\
  \hline
 TTH& 	7/10  & 1/2 & 	5/8  &  1/4  \\
  \hline
 TTT& 	1/2  &  3/10 & 	5/12  & 1/8 \\
  \hline
\end{tabular}
\end{center}
\caption{Bob's probability of winning.}
\label{table:PABobProb}
\end{table}

Table~\ref{table:PABobBestChoice} below displays the best choice for Bob given Alice's string and the odds of Bob winning for that choice. Once again, we only need to consider Alice's strings which start with H.

\begin{table}[h!]
\begin{center}
\begin{tabular}{ |c|c|c|} 
  \hline
   Alice's string  & Bob's best choice & Bob's odds \\
  \hline
 HHH& 	THH  & 	7 to 1 \\
  \hline
 HHT& 	THH  & 	3 to 1	  \\
  \hline
 HTH& 	HHT & 	2 to 1  \\
  \hline
 HTT& 	HHT   &	2 to 1   \\
	\hline
\end{tabular}
\end{center}
\caption{Bob's best choice and odds.}
\label{table:PABobBestChoice}
\end{table}

Figure~\ref{fig:Penney} shows Bob's best choice. The arrow points from Alice's string to Bob's best choice.

\begin{figure}[h!]
\centering
\includegraphics[scale=0.5]{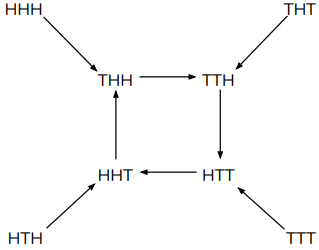}
\caption{Bob's best choice given Alice's string.}
\label{fig:Penney}
\end{figure}

Interestingly, there is an algorithm Bob can employ to choose his string once provided Alice's string such that he has the optimal probability of winning in Penney's game. Namely, given Alice's string of Hs and Ts, Bob takes the second letter from Alice's string, reverses it (turns it to H if it's T and vice versa), and adds the first $n-1$ elements of Alice's string to it (see \cite{GO}). From paper \cite{F}, we know that this optimal string is unique.

To demonstrate this with an example, suppose Alice chooses the string HTT. According to the algorithm described above, Bob chooses HHT, which creates a win ratio of 2:1 against Alice. If Alice chooses HHT this time, then Bob selects THH to stand victorious with the odds of 3:1. Hence, one may conclude that THH is a better choice than both HTT and HHT. However, this conclusion is incorrect; THH, for all practical purposes, is the same as HTT. Thus, we attain a non-transitive cycle: HHT beats HTT, which beats TTH, which beats THH, which beats HHT.
	
Since there is no best string that Alice can select, this game gives Alice a major disadvantage. Given Alice's string, Bob may construct his own such that he invariably bears the higher likelihood of victory. Now, with Alice and Bob, we seek to construct a coin game that, once and for all, gives the two equitable odds.
	
Here is the list of games we invent and discuss in this paper:

\begin{enumerate}
	\item \textbf{The Head-Start Penney's game. Head-start} We give Alice an edge by giving her a head-start of one flip. We discuss this coin game in detail in Section~\ref{sec:eff}. Bob's expected wait time increases by 1 and Alice's chances of winning are improved. However, Bob has the advantage and, given her string, is capable of selecting his own such that his probability of winning is better than a half. Interestingly, given Alice's string, Bob's optimal string in this game is the same as Bob's optimal string in Penney's game.
	\item \textbf{The Post-a-Bobalyptic Penney's game.} We discuss this game extensively in Section~\ref{sec:efl}. In this variant of Penney's game, Alice is allowed an additional toss at the end if Bob's string appears before hers does. If her string appears with this extra toss, it is counted as a tie. Otherwise, Bob wins. For every choice Alice may make, her worst-case odds are better than in Penney's game. Additionally, if Alice picks a good string, this game is fair for $n=3$. 
	\item \textbf{The Second-Occurrence Penney's game.}  We explore and discuss this coin game in Section~\ref{sec:so}. Here, the first player whose string occurs twice wins. This game does improve the odds in favor of Alice, but with a few exceptions. 
	\item \textbf{The Two-Coin game.} In this game, which we analyze in Section~\ref{sec:tc}, Alice and Bob flip separate coins. Whoever sees their string appear first in their own series of flips wins. The Two-Coin game differs from other games in that the winning probabilities only depend on expected wait time as the two strings cannot interact with one another. This game is fair for both players.
	\item \textbf{The No-Flippancy game.} We discuss this game in detail in Section~\ref{sec:nf}. This game, being deterministic, is rather distinct from the others here, as it does not involve a coin. Alice and Bob decide on their turn what the result of a coin flip is according to simple rules. Alice has the first turn, and they alternate turns afterwards. The nature of this variant is such that the outcome of each game is always either a win, a loss, or an infinite game. Unfortunately, it is not enough for her to have a fair chance in this game. Interestingly, for $n=3$, Bob can follow the algorithm to obtain an optimal string in Penney's game to pick a victorious string in the No-Flippancy game. As this game has numerous highly intriguing properties, we have written an in-depth paper about it \cite{NoFlippancy}.
	\item \textbf{The Blended game.} This game is a friendly mixture of the No-Flippancy and Penney's game. We cover it in Section~\ref{sec:c}. When Alice's and Bob's wanted outcomes coincide, that is the outcome they receive, similar to the No-Flippancy Game. If not, they flip a coin. Once again, whoever's string appears first is triumphant. One interesting feature of this game is that it also has a non-transitive cycle, which is of length 6.
\end{enumerate}

It is worth noting that, in Penney's game, the odds of a string $S_1$ winning against a string $S_2$ are independent of which player selects which string. This is due to the symmetry present in the game with respect to Alice and Bob.
	
However, this property is not true in three of the game variants listed above: the Head-Start Penney's game, the Post-a-Bobalyptic Penney's game, and the No-Flippancy game. Nevertheless, the remaining three games --- the Second-Occurrence game, Two-Coins game, and Blended game --- retain this characteristic.

Interestingly, the Second-Occurrence game has a non-transitive cycle of strings congruent to that of Penney's game in Figure \ref{fig:Penney}. The Blended game has a non-transitive cycle of length six, which furthermore surprised us.

We cover some general theory in Section~\ref{sec:sometheory}. Specifically, we discuss a way to measure how self-overlapping a string is, and we use this knowledge to explain the formula for the odds for Penney's game. Additionally, we mention the probability that a string appears for the first time in a given number of flips.

Below are conventions we utilize throughout the paper.

\begin{itemize}
\item We assume that Alice's and Bob's string are different, as most games would end in a tie otherwise.  
\item Unless otherwise noted, we assume that Alice's string starts with H, as the results for strings starting with T can be extended by symmetry.
\item Alice and Bob choose strings of the same lengths. We denote this common length as $n$.
\item In all the tables with winning probabilities, unless otherwise noted, the top row shows Alice's choices, the left column shows Bob's choices, and we show Bob's probabilities of winning. 
\end{itemize}

Alice and Bob commence their journey with the simplest variation of Penney's game: the Head-Start Penney's game.

\section{The Head-Start Penney's game}\label{sec:eff}

Alice and Bob, deeply unsatisfied by the unfairness inherent to Penney's Game, mutually concur to even the odds somewhat. Bob graciously agrees to play the Head-Start Penney's game instead. In this variation of Penney's game, the first toss counts only towards Alice's string. Just as with Penney's game, the two players cannot tie.
	
Before they play, the two companions wish to investigate how much, precisely, the game leverages the odds in Alice's favor.

``Is there a method by which we may derive how much my probability of winning has improved in comparison to the standard Penney's game?'' Alice inquires.
	
``The first toss aids you winning only if you're victorious by immediately receiving your string in the initial $n$ tosses. Otherwise, the first toss doesn't affect the game,'' Bob realizes.
	
Excitedly, Alice adds, ``I see! If I fail to win immediately, the game will proceed as Penney's game would. The probability of a win occurring for me right off the bat is $\frac{1}{2^{n}}$.''

With this new knowledge, they denote the probability of Alice winning in Penney's game, given her and Bob's strings, as $p$. They now know that if Alice does not receive a win immediately, which is the case with probability $1-\dfrac{1}{2^{n}}$, the game proceeds as Penney's game would. So, Alice's total winning chance with the same strings selected must be $\frac{1}{2^n} + (1-\frac{1}{2^n})p$. This can be rewritten as $p+(1-p)\frac{1}{2^{n}}$.
	
Therefore, Bob's probability is $1 - \frac{1}{2^n} - (1-\frac{1}{2^n})p = (1-p)(1-\frac{1}{2^n})$. His chance for a wonderful win in Penney's game is $1 - p$, so we can safely infer that his probability of victory in the Head-Start Penney's game is directly proportional to his probability of victory in Penney's game.

``The Head-Start Penny's game indeed provides me with better odds than Penney's game does. Notably, the larger $p$ is and the smaller $n$ is, the greater the increase in my probability of winning in this game is compared to Penney's game!'' Alice gleefully remarks.
		
``That, in turn, leads to a perhaps even more interesting statement,'' Bob comments. ``Given your string, my optimal string in Penney's game must also be the optimal string for me in this game, since the probabilities are directly proportional! Quite fascinating.''
		
``Precisely,'' we interject in agreement. ``Suppose, for example, $n=3$. Then, Alice, your chance of winning is $\frac{1+7p}{8}$, and Bob's chance of winning is $\frac{7(1-p)}{8}$ from what we established together.''
		
And thus, using the winning probabilities for Penney's game for $n=3$ from Table~\ref{table:PABobProb}, we attain Table~\ref{table:FTBobProb} below, which delineates Bob's probabilities of winning in the Head-Start Penney game.

\begin{table}[h!]
\begin{center}
    \begin{tabular}{|c|c|c|c|c|}
    \hline
    & HHH & HHT & HTH & HTT \\
    \hline
    HHH & * & 7/16 & 7/20 & 7/20 \\
    HHT & 7/16 & * & 7/12 & 7/12 \\
    HTH & 21/40 & 7/24 & * & 7/16 \\
    HTT & 21/40 & 7/24 & 7/16 & * \\
    THH & 49/64 & 21/32 & 7/16 & 7/16 \\
    THT & 49/96 & 21/64 & 7/16 & 7/16 \\
    TTH & 49/80 & 7/16 & 35/64 & 7/32 \\
    TTT & 7/16 & 21/80 & 35/96 & 7/64 \\
    \hline
    \end{tabular}
\end{center}
\caption{Bob's probabilities of winning.}
\label{table:FTBobProb}
\end{table}

Table~\ref{table:FTBobsodds} shows Bob's optimal string choice and odds of winning given Alice's string. As Alice mentioned previously, the optimal string for Bob is the same here as it is in Penney's game. This signifies that Figure~\ref{fig:Penney} represents this game as well. Nevertheless, the odds in the Head-Start Penney's game are certainly more favorable for Alice in comparison to Penney's game. Despite this, Bob still has an advantage because he can always pick a string with favorable odds.

\begin{table}[h!]
\begin{center}
    \begin{tabular}{|c|c|c|}
    \hline
    Alice's string & Bob's Best choice & Bob's odds \\
    \hline
    HHH & THH & 49 to 15 \\
    \hline
    HHT & THH & 21 to 11 \\
    \hline
    HTH & HHT & 7 to 5 \\
    \hline
    HTT & HHT & 7 to 5 \\
    \hline
    \end{tabular} 
\end{center}
\caption{Bob's best choice and odds.}
\label{table:FTBobsodds}
\end{table}

As Bob is delayed by a toss, the expected wait time for his string increases by 1 flip. In this game, as in Penney's game, the string bearing a longer wait time may still have a higher success rate. In Penney's game, this only occurs for strings where $n\geq 4$. In the Head-Start Penney's game, this can occur for $n=3$ as well. All such cases for strings of length 3 are displayed in Table~\ref{table:Wvwt}, where ``Winner'' denotes the player with the advantage.

\begin{table}[h!]
\begin{center}
    \begin{tabular}{ | c | c | c | c |}
    \hline  
    Winner & Alice & Bob & Wait Time \\
    \hline
    Alice & HTH & THH & (10, 9)\\
    Bob & HHT & THH & (8, 9) \\
    Alice & HTH & HTT & (10, 9) \\
     Bob & HTT & HHT & (8, 9) \\ 
    Alice & HHH & HHT & (14, 9) \\
    \hline
    \end{tabular}
\end{center}
\caption{Winning versus wait time.}
\label{table:Wvwt}
\end{table}

The last line in Table~\ref{table:Wvwt} particularly piques interest due to the incredibly disparate expected wait times contradicting the true winner. This phenomenon continues and magnifies for greater string lengths. If Alice selects a string comprised uniformly of Hs and Bob chooses a string congruent to Alice's except for the last term, Alice has the greater probability of winning despite her string's expected wait time being almost twice as large as Bob's string's expected wait time.
	
Alice and Bob fail to reach their goal with the Head-Start Penney's game, yet are gratified by the knowledge that allowing Alice an extra flip at the beginning did indeed help Alice. ``What if Alice is provided with an extra flip at the end of Penney's game instead?'' they wonder pensively.

\section{The Post-a-Bobalyptic Penney's game}\label{sec:efl}

This game is distinguished from Penney's game by the fact that, if she loses, Alice receives an extra toss at the very end. If this toss allows her string to appear, then it the outcome is a tie.
	
``Well,'' Bob begins, ``if I attain my string first in the sequence of tosses, then you get an extra flip. However, that additional flip can only be of use if the last $n-1$ terms of my string coincide with the first $n-1$ terms of your string.''
	
``True,'' Alice remarks. ``However, your algorithm for obtaining the optimal string in Penney's game creates this manner of a special pair of strings.''
	
As an example, suppose Alice selects the string HHH. To maximize his advantage in Penney's game, Bob would select THH by the best-odds algorithm. However, this gives Alice a chance to tie things up with the extra toss at the end. 

Bob nods in agreement. ``In this game, whenever I use this `optimal' sequence, there must be a $\frac{1}{2}$ probability that you will receive your string's last character and we tie. However, if I select some other string, my probability of victory must be equivalent to that in Penney's game.''

As Alice and Bob wish to model their probabilities in this novel game mathematically, we provide them with a basis. Suppose, for some set of strings chosen by Alice and Bob, Alice comes out the winner with the probability $p(A)$ in Penney's game and with the probability $p'(A)$ in the Post-a-Bobalyptic Penney's game. Correspondingly, suppose Bob is triumphant against Alice in Penney's game with the probability $p(B)$ and in the Post-a-Bobalyptic Penney's game with probability $p'(B)$. We furthermore denote by $p'(T)$ the probability that Alice and Bob obtain a tie in the Post-a-Bobalyptic Penney's game. We therefore have $p(A) + p(B) = 1$ and $p'(A) + p'(B) +p'(T) = 1$.
	
Alice and Bob conclude that if the first $n-1$ elements of Alice's string fail to match with the last $n-1$ elements of Bob's string, then 
\[p'(T) = 0 \quad p'(A) = p(A) \quad p'(B) = p(B).\]
Otherwise, 
\[p'(T) = \frac{p(B)}{2} \quad p'(A) = p(A) \quad p'(B) = \frac{p(B)}{2}.\]

This illustrates that, given Alice's string, if Bob chooses his string by performing the best-odds algorithm, then his odds of winning in the Post-a-Bobalyptic Penney's game are $\frac{p(B)}{2p(A)}$. As described, this is half his odds in Penney's game. Therefore, another string might be a better choice for Bob in the Post-a-Bobalyptic Penney's game. In any case, his odds are not reduced by more than twice.

Table~\ref{table:ETpoda} displays probabilities of different game outcomes for strings of length 3. The first term represents the probability of Bob winning, the second one is the probability of Alice winning, and the last one is the probability of a tie.

\begin{table}[h!]
\begin{center}
\begin{tabular}{ |c|c|c|c|c| } 
  \hline
     & HHH & HHT & HTH & HTT \\
  \hline
 HHH& *  & 1/4, 1/2, 1/4  &2/5, 3/5, 0  &  2/5, 3/5, 0 \\
  \hline
 HHT& 1/2, 1/2, 0  & *  & 1/3, 1/3, 1/3  & 1/3, 1/3, 1/3  \\
  \hline
 HTH&  3/5, 2/5, 0 & 1/3, 2/3, 0  &  * &  1/2, 1/2, 0 \\
  \hline
 HTT& 3/5, 2/5, 0   & 1/3, 2/3, 0   & 1/2, 1/2, 0  & *  \\
  \hline
 THH& 	7/16, 1/8, 7/16 &  3/8, 1/4, 3/8 & 1/2, 1/2, 0   & 1/2, 1/2, 0  \\
  \hline
 THT& 7/12, 5/12, 0  & 3/8, 5/8, 0  & 1/4, 1/2, 1/4  & 1/4, 1/2, 1/4 \\
  \hline
 TTH& 7/10, 3/10, 0  &  1/2, 1/2, 0 & 5/8, 3/8, 0  &  1/4, 3/4, 0  \\
  \hline
 TTT& 1/2, 1/2, 0  &  3/10, 7/10, 0 & 5/12, 7/12, 0  &  1/8, 7/8, 0 \\
  \hline
\end{tabular}
\end{center}
\caption{Probabilities of different outcomes for the Post-a-Bobalyptic Penney's game.}
\label{table:ETpoda}
\end{table}

Table~\ref{table:ETodds} shows the odds of Bob winning.

\begin{table}[h!]
\begin{center}
\begin{tabular}{ |c|c|c|c|c| } 
  \hline
     & HHH & HHT & HTH & HTT \\
  \hline
 HHH& *   & 1 to 2  & 2 to 3 &  2 to 3 \\
  \hline
 HHT& 1 to 1  & *  & 1 to 1  & 1 to 1  \\
  \hline
 HTH&  3 to 2 & 1 to 2 &  * &  1 to 1 \\
  \hline
 HTT& 3 to 2  & 1 to 2  & 1 to 1 & *  \\
  \hline
 THH& 	7 to  2&  3 to 2 & 1 to 1   & 1 to 1 \\
  \hline
 THT& 7 to 5  & 3 to 5  & 1 to 2  & 1 to 2 \\
  \hline
 TTH& 7 to 3  &  1 to 1 & 5 to 3  &  1 to 3  \\
  \hline
 TTT& 1 to 1  &  3 to 7 & 5 to 7  &  1 to 7 \\
  \hline
\end{tabular}
\end{center}
\caption{Bob's odds of victory in the Post-a-Bobalyptic Penney's game.}
\label{table:ETodds}
\end{table}

As seen in the table, this game is now a fair game for $n=3$. If Alice picks HTT for her string, at best, Bob can only tie with Alice. However, for any other string, Bob would still have an advantage.

Table~\ref{table:ETbest} shows the best choice for Bob and the odds of Bob's winning.

\begin{table}[H]
\begin{center}
\begin{tabular}{ |c|c|c|} 
  \hline
   Alice  & Bob & Odds \\
  \hline
 HHH& 	THH  & 	7 to 2 \\
  \hline
 HHT& 	THH  & 	3 to 2	  \\
  \hline
 HTH& 	TTH & 	5 to 3  \\
  \hline
 HTT& 	HHT, HTH, THH &	1 to 1   \\
  \hline
\end{tabular}
\end{center}
\caption{Bob's best choice and odds.}
\label{table:ETbest}
\end{table}

Figure~\ref{fig:LF} shows Bob's best choice. The arrow points from Alice's string to Bob's best choice.

\begin{figure}[H]
\centering
\includegraphics[scale=0.5]{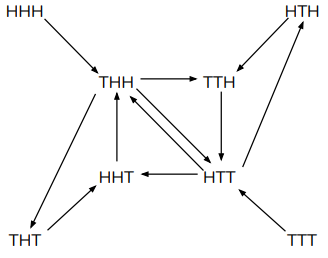}
\caption{Bob's best choice.}
\label{fig:LF}
\end{figure}

As is displayed in the preceding data tables, the optimal string choice for Bob in Penney's game remains the optimal choice in the Post-a-Bobalyptic Penney's game for $n=3$. However, in this game, it may share the mantle of `optimal string' with other strings.

We observed that in Penney's game, a string with a greater expected wait time may have an advantage. This is true in the Post-a-Bobalyptic Penney's game as well. We must have $n\geq 4$ for this to occur. Table~\ref{table:ETslwt} shows all such possibilities where $n=4$.

\begin{table}[H]
\begin{center}
    \begin{tabular}{ | c | c | c | c |}
    \hline  
    Winner & Alice & Bob & Wait Time \\
    \hline
    Alice & HTTH & TTHH & (18,16)\\
    \hline
    Alice & HTHT & THTT & (20,18)\\
    \hline
    Alice & HTHH & THHH & (18,16)\\
    \hline
    Bob & HHTT & HTHH & (16,18) \\ 
    \hline
    Alice & HTHH & HHTT & (18,16) \\
    \hline
    \end{tabular}
\end{center}
\caption{Pairs of strings for $n=4$ where the string with the longer wait time has higher winning probabilities.}
\label{table:ETslwt}
\end{table}

The penultimate line garners considerable interest. The Post-a-Bobalyptic game provides a greater advantage to Alice than Penney's game does in certain cases. However, Bob may win despite having a string with a longer expected wait time.
	
As such, Alice and Bob continue their quest for a firmly fair game.

\section{The Second-Occurrence Penney's game}\label{sec:so}

Alice and Bob have seen that adding a flip in Alice's favor has not quite completely worked out as intended. While there was certainly a slight improvement in her odds, the games thus far have failed to be fair. Instead, Alice suggests a new game --- the \textit{Second-Occurrence Penney's game}. Here, a player is declared the victor once their string appears the second time in the sequence of coin flips.

As an example, suppose Alice's string is HHH and Bob's string is HHT. The game output HHTHHHH is a win for Alice, as HHH appears twice whereas Bob's string only appears once.

As with Penney's game, there are no ties possible in the Second-Occurrence Penney's game. Additionally, if Alice and Bob choose complementary strings, then, by symmetry, they both win with an equal probability.

We define a \textit{self-overlapping string} as follows: one can fit two occurrences of a string of length $n$ into less than $2n$ tosses. These strings might have better odds in this game than in the original Penney's game. For example, HHH is an unfavorable string selection for Alice in Penney's game. However, in the Second-Occurrence Penney's game, Alice has a 1/2 chance of attaining the second occurrence of this string immediately after the first occurrence if the game output is HHHH, which improves her chances.

We used a computer to calculate the probabilities of different outcomes for $n=3$. The results are presented in Table~\ref{table:SOBwp}. As before, the numbers below are Bob's winning probabilities.

\begin{table}[h!]
\begin{center}
\begin{tabular}{|c|c|c|c|c|}
\hline
 & HHH & HHT & HTH & HTT\\
\hline
HHH & * & 1/2 & 56/125 & 56/125\\
\hline
HHT & 1/2 & * & 16/27 & 16/27\\
\hline
HTH & 69/125& 11/27& * & 1/2\\
\hline
HTT& 69/125& 11/27& 1/2&*\\
\hline
THH& 11/16& 3/4& 1/2 & 1/2\\
\hline
THT& 59/108 & 7/16 & 1/2& 1/2\\
\hline
TTH& 151/250 & 1/2& 9/16& 1/4\\
\hline
TTT& 1/2 & 99/250& 49/108& 5/16\\
\hline
\end{tabular}
\end{center}
\caption{Bob's winning probabilities in the Second-Occurrence Penney's game}
\label{table:SOBwp}
\end{table}

Table~\ref{table:SOBodds} displays Bob's optimal string choice for different string selections made by Alice and the corresponding odds. The best choice for Bob for $n=3$ is the same as in Penney's game and is represented in Figure~\ref{fig:Penney}.

\begin{table}[H]
\begin{center}
\begin{tabular}{ |c|c|c| } 
 \hline
 Alice's string & Bob's best choice & Bob's odds \\ 
 \hline
 HHH & THH & 11 to 5 \\ 
 \hline
 HHT & THH  & 3 to 1 \\ 
 \hline
 HTH & HHT & 16 to 11 \\
 \hline
 HTT & HHT & 16 to 11 \\
 \hline
\end{tabular}
\end{center}
\caption{Bob's best choice and odds.}
\label{table:SOBodds}
\end{table}

We now look at this table and discuss our observations.

\begin{enumerate}
\item If Alice and Bob select complementary strings, the probability of either winning is 1/2, as expected.

\item We assumed that the Second-Occurrence Penney's game is fairer for Alice than the original Penney's game is. The probabilities are not altered between Penney's game and this game in cases where both players have a winning probability of 1/2. This result is not surprising. The probability 1/2 is achieved for pairs of complementary strings and pairs of strings that differ only in the last character.

\item For the string pair HHT and THH (and, similarly, the complementary pair TTH and HTT), the probabilities did not change between this game and Penney's game. The probability theory is rather tricky here. We can show not only for the Second-Occurrence Penney's game but for any $k$th-Occurrence Penney's game as well that the probability for this pair of strings fails to change. Indeed, if we have a long string of tosses, HHT and THH must alternate. THH always appears between two occurrences of HHT and, likewise, one HHT always occurs between two occurrences of THH. Therefore, whoever obtains the first occurrence of their string wins the any $k$th-Occurrence Penney's game.

\item For the remaining pairs of strings, the probabilities are closer to 1/2 than in Penney's game.
\end{enumerate}

\section{Some theory}\label{sec:sometheory}

Penney's game has been extensively studied. The theory presented in this section is available in many places, including \cite{BCG,C,F,MGSA,MG,GO,HN,WP,RV}. The theory is built not only for strings of coin tosses with two outcomes but for words in any alphabet. We do not discuss general formulas, but rather assume that the size of our alphabet is 2.

\subsection{Autocorrelation polynomial and Conway's Leading Number}

The autocorrelation vector of a word $w = w_{1}\dots w_{n}$, where $w_i$ is the $i$-th character in the word, is $c=(c_{0},\dots ,c_{n-1})$, with $c_{i}$ being 1 if the prefix of length $n - i$ equals the suffix of length $n-i$, and with $c_{i}$ being 0 otherwise. That is, value $c_{i}$ indicates whether $ w_{i+1}\dots w_{n}=w_{1}\dots w_{n-i}$.

For example, the autocorrelation vector of HHH is $(1,1,1)$ since all the prefixes are equal to all the suffixes. The autocorrelation vector of THH is $(1,0,0)$ since no proper prefix is equal to a proper suffix. Finally, the autocorrelation vector of HHTTHH is $(1,0,0,0,1,1)$.

Note that $c_{0}$ is always equal to 1 since the prefix and the suffix of length $n$ are both equal to the word $w$. Similarly, $c_{n-1}$ is 1 if and only if the first and the last letters are the same. If $c_1 =1$, then all letters in the word are the same. It follows that if $c_1 =1$, then for any $i$ we have $c_i = 1$.

If $c_0$ is the only value that is 1, we call the word $w$ a \textit{non-self-overlapping word}.

The autocorrelation polynomial of $w$ is defined as $c_w(z)=c_{0}z^{0}+\dots +c_{n-1}z^{n-1}$. It is a polynomial of degree at most $n-1$.

For example, the autocorrelation polynomial of HHH is $1+z+z^{2}$ and the autocorrelation polynomial of THH is 1. Finally, the autocorrelation polynomial of HHTTHH is $1+z^{4}+z^{5}$.

The Conway Leading Number is the integer value of the correlation vector viewed as a string of zeros and ones and evaluated in binary. Thus, the Conway Leading Number for HHH is 7, the Conway Leading Number for THH is 4, and the Conway Leading Number for HHTTHH is 35. We denote the Conway leading number of a word $w$ as $C_w$.

We can express the Conway Leading Number through the autocorrelation polynomial:
\[C_w = 2^{n-1}c_w\left(\frac{1}{2}\right).\]

\subsection{Probabilities}

Suppose we have a string $A$ of length $n$. We call a string $S$ that ends with $A$ an \textit{A-first-timer} if the last $n$ characters of $S$ is the only occurrence of $A$ in $S$. We denote the number of $A$-first-timers of length $i$ as $a_i(A)$ for $i \geq 0$. We call the sequence $a_1(A)$, $a_2(A)$, $\ldots$ and so on an \textit{$A$-first timer sequence}. 

Let $p_i(A)$ denote the probability that we first see the string $A$ after the $i$-th toss of a coin. The probabilities $p_i(A)$ and integers $a_i(A)$ are connected:
\[p_i(A) = \frac{a_i(A)}{2^i}.\]

For example, we know that $a_i(HT) = i-1$. Thus, $p_i(HT) = \frac{i-1}{2^i}$.

The generating function for the  $w$-first-timer sequence, where word $w$ has length $k$, is known \cite{GO}:
\[G(z) = \frac {z^{k}}{z^{k}+(1-2z)c_w(z)}.\]

If we plug in $z = \frac{1}{2}$, we get 1, which is the probability that the string ever occurs.

This generating function allows us to derive the formula for the expected wait time.

\subsection{The expected wait time}

The expected wait time of string $A$ can be calculated as 
\[\sum_0^\infty ip_i = \sum_0^\infty ia_i(A) \frac{1}{2^i}.\]

Given that $G(z)$ above is the generating function for the sequence $a_i(A)$, the generating function for the sequence $ia_i(A)$ is $zG'$, which equals
\[\frac {z(kz^{k-1}(z^{k}+(1-2z)c(z)) - z^k(kz^{k-1} -2c(z) + (1-2z)c'(z)))}{(z^{k}+(1-2z)c(z))^2}.\]

Minor manipulation produces
\[\frac {kz^{k}((1-2z)c(z)) + 2z^{k+1}c(z) - z^{k+1}(1-2z)c'(z)))}{(z^{k}+(1-2z)c(z))^2}.\]

When we plug in $z = \frac{1}{2}$, we get
\[\frac {2(\frac{1}{2})^{k+1}c(z)}{(\frac{1}{2})^{2k}} = 2^kc(z).\]

We can see that the expected wait time for a string is its Conway leading number times 2, which is a well-known fact. 

Here are examples of the wait time for strings of length 3: HHH --- 14; HTH --- 10; HHT and HTT --- 8.

We can see some patterns: If all letters are the same, then the wait time for a string of length $n$ is $2^{n+1} - 2$. This is the longest wait time for strings of length $n$. Strings that are reverses of each other have the same wait time as they have the same autocorrelation polynomial. The shortest wait time is $2^n$, which is provided by non-self-overlapping strings. We see that the wait time for some strings is almost twice the wait time for others.

\subsection{Examples}

Table~\ref{table:smallstrings} shows the examples of the first-timer sequences for strings of length up to 3 and some especially interesting examples for strings of length 4.

We group strings with the same autocorrelation polynomial, that is, strings with the same Conway Leading Number (CLN), into one line.

\begin{table}[h!]
\begin{center}
\begin{tabular}{ |c|c|c| } 
 \hline
 String & CLN &  Sequence\\ 
\hline
  H,T & 1& 1,1,1,1,1,1,1,1,1\\
\hline
HH, TT & 11 & 0,1,1,2,3,5,8,13,21 \\
\hline
HT, TH & 10 & 0,1,2,3,4,5,6,7,8\\
\hline
 HHH, TTT & 111 & 0,0,1,1,2,4,7,13,24 \\ 
\hline
 HHT, HTT, THH, TTH & 100 & 0,0,1,2,4,7,12,20,33\\ 
 \hline
HTH, THT & 101 & 0,0,1,2,3,5,9,16,28\\
\hline
HHHH, TTTT	& 1111 & 0,0,0,1,1,2,4,8,15\\
\hline
HHHT,HHT,HTTT,TTTH,TTHH,THHH& 1000 & 0,0,0,1,2,4,8,15,28\\
\hline
\end{tabular}
\end{center}
\caption{First-timer sequences for some strings of length up to 4.}
\label{table:smallstrings}
\end{table}

Table~\ref{table:smallstringsOEIS} shows, given a Conway Leading Number, the sequence number in the OEIS and its linear recurrence. The sequence in the OEIS is often shifted. We also provide the name of the sequence if it is short.

\begin{table}[h!]
\begin{center}
\begin{tabular}{ |c|c|c|c| } 
 \hline
 CLN & OEIS & Name & Recurrence \\ 
\hline
 1 & A000012 & Only 1s& $a_{n} = a_{n-1}$\\ 
\hline
 11 & A000045 & Fibonacci Numbers& $a_{n} = a_{n-1} + a_{n-2}$\\ 
 \hline
10 & A001477& Whole numbers& $a_{n} = 2a_{n-1} - a_{n-2} $\\
\hline
111 & A000073& Tribonacci Numbers&  $a_{n} = a_{n-1} + a_{n-2} + a_{n-3}$\\
\hline
100 & A000071& Fibonacci Numbers $-1$ & $a_{n} = 2a_{n-1}  a_{n-3}$\\
\hline
101 & A005314& (No short description) & $a_{n} = a_{n-1} + a_{n-2} + a_{n-4}$\\
\hline
1111 &A000078 & Tetranacci numbers & $a_{n} = a_{n-1} + a_{n-2}+ a_{n-3} + a_{n-4}$\\
\hline
1000 & A008937& Partial sums of Tribonacci numbers&$a_{n} = 2a_{n-1} - a_{n-4}$\\
\hline
\end{tabular}
\end{center}
\caption{First-timer sequences for small strings in the OEIS.}
\label{table:smallstringsOEIS}
\end{table}

\textbf{Example HHT.}  Consider the string HHT of length 3. The corresponding first-timer sequence is:
\[0,\ 0,\ 1,\ 2,\ 4,\ 7,\ 12, \ldots,\]
with the recurrence $a_i = 2a_{i-1} - a_{i-3}$ for $i > 2$.

This is sequence A000071 in the OEIS. We see that $a_i(HHT) = F_{i+1} - 1$, where $F_i$ is the $i$-th Fibonacci number. This is also sequence of partial sums of the Fibonacci sequence shifted one place to the right. One of the descriptions of this sequence in the OEIS says that this sequence is the number of 001-avoiding binary words of length $i - 3$. This is exactly our sequence. Indeed, removing HHT from the end of HHT-first-timers makes a bijection between our strings and strings avoiding HHT. The latter is the same as the number of binary strings avoiding 001. The generating function for the HHT-first-timer sequence is 
\[\frac{x^2}{(1-x-x^2)(1-x)}.\]

\textbf{Strings with one letter.}  Consider strings that consist of the same letter. In our table, they correspond to the following sequences: Only 1s, Fibonacci numbers, Tribonacci numbers, and Tetranacci numbers. They form a clear pattern. For strings of length $n$, each next term of the sequence is the sum of $n$ previous terms. The sequence of only 1s belongs to this pattern. One might call it the ``mononacci'' numbers.

\textbf{Non-self-overlapping strings.}  Consider non-self-overlapping strings, that is, strings with Conway Leading Numbers that are powers of two. In our table, they correspond to the following sequences:

\begin{itemize}
\item Only 1s 
\item Whole numbers
\item Fibonacci numbers $-1$ 
\item Partial sums of Tribonacci numbers
\end{itemize}

At first glance, there is no clear pattern. Then, one might remember the following famous identity for the Fibonacci numbers:
\[\sum_{i=0}^n F_i = F_{n+2}-1.\]
That is, the Fibonacci sequence is the sequence of partial sums of the Fibonacci numbers shifted by two terms. Now we can see a pattern. The sequence of whole numbers is also the sequence of partial sums of the sequence of all ones (our mononacci sequence). For a non-self-overlapping string of length $n$, the sequence is the partial sums of the $n-1$-nacci numbers.

\subsection{Game results}

To give the formula for winning probabilities, we need to extend the notion of Conway Leading Numbers. Conway Leading Numbers quantify the overlaps of a word with itself. However, we can also quantify the overlap of a word with another word. We only define the correlation of two words of the same length $n$, but it also can be defined for words of different lengths.

The correlation vector of two words $w_1$ and $w_2$ is $c_{w_1,w_2}=(c_{0},\dots ,c_{n-1})$, with $c_{i}$ being 1 if the prefix of length $n - i$ of $w_2$ equals the suffix of length $n-i$ of $w_1$, and with $c_{i}$ being 0 otherwise.

The correlation vector defines the correlation polynomial, the same way as it is defined by one word. It also defines the Conway Leading Number for two words, which we denote as $C_{w_1,w_2}$.

If $A$ is Alice's string and $B$ is Bob's string, then the formula for Bob's odds in Penney's game is:
\[\frac{C_{A,A}-C_{A,B}}{C_{B,B}-C_{B,A}}.\]

\textbf{Example HHT and THT.} For $i=0$, we see that HHT is different from THT, so $c_0 = 0$. For $i=1$, we see that HT is different from TH, meaning that $c_1 = 0$. For $i=2$ we see that T is the same as T, so $c_2 = 1$. Therefore, $c_{HHT,THT} = (0,0,1)$ and $C_{HHT,THT} = 001_2 = 1$. Similarly, we get that $C_{HHT,HHT}=4$, $C_{THT,THT} = 5$, and $C_{THT,HHT} = 0$. This means that Bob's comparative odds are $\frac{4-1}{5-0} = \frac{3}{5}$ and his probability of winning is $\frac{3}{8}$.

Based on this formula, the strategy for Bob to pick the best string when Alice's string is given is the following \cite{F}. Bob takes the second character in Alice's string and flips it. Then, he appends to it Alice's prefix of size $n-1$. For example, if Alice's string is HHT, Bob flips the second character H and uses T as his first character. Then he appends HH. Thus Bob's best choice is THH. 

We are now ready to describe a new game that might be considered a fair game.

\section{The Two-Coin game}\label{sec:tc}

Alice and Bob, quite dismayed by their repeated failure to attain a truly fair game, decide that a single coin will not cut it; they unveil their second and last coin, furiously brainstorming to come up with a game necessitating the involvement of two coins.

At last, their conjoined brainpower culminates to the simplistic yet elegant Two-Coin game. The two players toss their coins side-by-side, thus simultaneously developing separate sequences of flip outcomes as opposed to the singular sequence in Penney's game. The first person to attain their string through the sequence of flips that they generate is the winner. In this game, a tie is possible when both players simultaneously get their string while flipping their own coins.

However, their effort has resulted in a profound revelation: this game may, for once, truly be fair. We show that the winning probabilities for Alice and Bob are solely dependent on the expected wait time rather than the interaction between the strings. 

\begin{proposition}
The probabilities of a win/loss/tie depend only on the expected wait times of Alice's and Bob's strings.
\end{proposition}

\begin{proof}
Alice's string is $A$ and Bob's string is $B$. Suppose that $p_i$ (correspondingly $q_i$) are the probabilities that Alice (correspondingly Bob) sees her (his) string for the first time at the $i$-th toss. These probabilities allow us to calculate the winning, losing, and tying probabilities for Alice.

Alice wins with probability 
\[\sum_{i=1}^{\infty}p_i(1 - \sum_{k=1}^i q_i).\]

Alice loses with probability
\[\sum_{i=1}^{\infty}q_i(1 - \sum_{k=1}^i p_i).\]

There is a tie with probability
\[\sum_{i=1}^{\infty}p_iq_i.\]

As the sequences $p_i$ and $q_i$ depend only on the autocorrelation of $A$ and $B$ respectively, it follows that the winning probabilities depend only on the autocorrelation of each sequence. In other words, they only depend on the Conway Leading Numbers, which are half the expected wait time.
\end{proof}

Note that these probabilities sum up to 1. That is, 
\[\sum_{i=1}^{\infty}p_i = \sum_{i=1}^{\infty}q_i = 1.\]
It follows that the probability of a win, loss, or a tie sum up to 1, or, equivalently, the probability of an infinite game is 0.

\textbf{Example.} Suppose Alice and Bob choose HT. As we showed before: $p_i(HT) = \frac{i-1}{2^i}$. Thus, the probability of a tie is 
\[\sum_{i=1}^{\infty}\frac{(i-1)^2}{4^i} = \sum_{i=0}^{\infty}\frac{1}{4}\cdot \frac{i^2}{4^{i}} .\]

The generating function of the squares sequence is $\frac{x(1 + x)}{(1 - x)^3}$. Plugging in $x=\frac{1}{4}$ and dividing by 4, we get $\frac{5}{27}$. That is, if both Alice and Bob choose HT or TH, then a tie occurs with probability $\frac{5}{27}$ and each of them wins with probability $\frac{11}{27}$.

We wrote a program to calculate probabilities for strings of length 3 that are now in Table~\ref{table:TCpoda}. The top row represents Alice's string's Conway Leading Number. The first column shows Conway's Leading Number for Bob's string. The numbers in each cell, from left to right, represent the probability of Bob winning, Alice winning, and a tie. 

\begin{table}[h!]
\renewcommand{\arraystretch}{2.5} 
\begin{center}
\begin{tabular}{ | c | c | c| c | } 
 \hline
   & 111 & 101 & 100 \\ 
 \hline
 111 & $\cfrac{435}{913}$,\ $\cfrac{435}{913}$,\ $\cfrac{43}{913}$ & $\cfrac{3289}{8691}$,\ $\cfrac{1643}{2897}$,\ $\cfrac{473}{8691}$  & $\cfrac{23327}{73057},\ \cfrac{45409}{73057},\ \cfrac{4321}{73057}$\\ 
 \hline
 101 & $\cfrac{1643}{2897},\ \cfrac{3289}{8691},\ \cfrac{473}{8691}$ & $\cfrac{487}{1045},\ \cfrac{487}{1045},\ \cfrac{71}{1045}$  & $\cfrac{22431}{55265},\ \cfrac{28673}{55265},\ \cfrac{4161}{55265}$ \\
 \hline
 100 & $\cfrac{45409}{73057},\ \cfrac{23327}{73057},\ \cfrac{4321}{73057}$ & $\cfrac{28673}{55265},\ \cfrac{22431}{55265},\ \cfrac{4161}{55265}$  & $\cfrac{377}{825},\ \cfrac{377}{825},\ \cfrac{71}{825}$\\ 
 \hline
\end{tabular}
\end{center}
\renewcommand{\arraystretch}{1} 
\caption{Probabilities of different outcomes.}
\label{table:TCpoda}
\end{table}

In Table~\ref{table:TCpodaApprox} we approximated the values from Table~\ref{table:TCpoda} to make it easier to compare.

\begin{table}[h!]
\begin{center}
\begin{tabular}{|c|c|c|c|}
\hline
& 111 & 101 & 100 \\\hline
111 & .48, .48, .047 & .38, .57, .055 & .32, .62, .059\\\hline
101 & .57, .38, .054 & .47, .47, .068 & .41, .52, .075\\\hline
100 & .62, .32, .059 & .52, .41, .075 & .46, .46, .086\\\hline
\end{tabular}
\end{center}
\caption{Approximate probabilities of different outcomes.}
\label{table:TCpodaApprox}
\end{table}

Figure~\ref{fig:TwoCoin} shows Bob's best choice. The arrow points from Alice's Conway Leading Number to Bob's best choice. 

\begin{figure}[h!]
\centering
\includegraphics[scale=0.6]{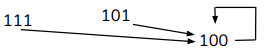}
\caption{Bob's best choice.}
\label{fig:TwoCoin}
\end{figure}

The results are less surprising for this game compared to Penney's game. The strings with shorter expected wait time give better winning odds. Thus, both Alice and Bob should choose the strings with Conway Leading Number 100: that is one of HHT, HTT, THH, or TTH.

We found our fair game. However, we also got tired of probabilities. So, we decided to invent a deterministic game.

\section{The No-Flippancy Game}\label{sec:nf}

Alice and Bob play their Two-Coin game from dusk to dawn, reveling in their discovery of a fair game. However, during an intermission in their game, Alice and Bob leave to complete some rudimentary chores. Bob hastily sweeps the two coins into his pocket before he departs so as to not forget them, blissfully unaware of the perfectly coin-sized hole at the bottom.

By the time Bob returns, the coins are nowhere to be found. Upon knowing this, Alice berates Bob about how his inattentiveness bungled the perfectly fair game.

Hoping to avoid Alice's scornful glare, Bob devises a coin game where they shall never need to flip a single coin --- thus designated the No-Flippancy game.

In this game, Alice chooses her string. Then Bob chooses his string of the same length. No one flips a coin. The outcome of the coin flip is decided deterministically. On their turn, each player looks at the current sequence of tosses and calculates the largest suffix, of length $i$, which matches the prefix of their string. Their choice is the character in position $i+1$ in their string.

We study the games where the length of Alice's and Bob's strings is 3.

Table~\ref{table:output} shows how the game proceeds. The top row shows Alice's choice and the left column shows Bob's choice. 

\begin{table}[H]
\begin{center}
\begin{tabular}{ |c|c|c|c|c| } 
  \hline
     & HHH 		&       HHT 	& HTH 		& HTT \\
  \hline
 HHH& 	HHH  		& 	HHT  	& HHTH	   	&  HHTHTHT...\\
  \hline
 HHT& 	HHH  		& 	HHT  	& HHT	  	& HHT		  \\
  \hline
 HTH& 	HTH		& 	HTH	& HTH  		& HTT	 \\
  \hline
 HTT& 	HTHTHT...   	& HTHTHT...   	& HTH  		& HTT	  \\
  \hline
 THH& 	HTHH 		&  HTHH	 	& HTH	   	& HTT	  \\
  \hline
 THT& 	HTHT  		& HTHT  	& 	HTH  	&  HTT	 \\
  \hline
 TTH& 	HTHTHT...  	& HTHTHT... 	& 	HTH  	&  HTT  \\
  \hline
 TTT& 	HTHTHT...  	& HTHTHT...  	& 	HTH  	& HTT \\
  \hline
\end{tabular}
\end{center}
\caption{The output of the game.}
\label{table:output}
\end{table}

Table~\ref{table:result} describes who wins and the number of turns in the game.

\begin{table}[h!]
\begin{center}
\begin{tabular}{ |c|c|c|c|c| } 
  \hline
     & HHH 		&       HHT 	& HTH 		& HTT \\
  \hline
 HHH& 	Tie, 3 		& 	Alice,3	& Alice, 4   	&  Tie, infinite \\
  \hline
 HHT& 	Alice, 3	& 	Tie, 3 	& Bob, 3  	& Bob, 3	  \\
  \hline
 HTH& 	Bob, 3		& 	Bob, 3	& Tie, 3	& Alice, 3 \\
  \hline
 HTT& 	Tie, infinite  	&Tie, infinite  & Alice, 3	& Tie, 3  \\
  \hline
 THH& 	Bob, 4 		&  Bob, 4 	& Alice, 3   	& Alice, 3  \\
  \hline
 THT& 	Bob, 4 		& Bob, 4 	& Alice, 3  	&  Alice, 3 \\
  \hline
 TTH& 	Tie, infinite 	& Tie, infinite	& Alice, 3  	&  Alice, 3  \\
  \hline
 TTT& 	Tie, infinite  	& Tie, infinite & Alice, 3  	& Alice, 3 \\
  \hline
\end{tabular}
\end{center}
\caption{The output of the game.}
\label{table:result}
\end{table}

We see that, in this game, if Alice chooses her string first, Bob can always choose his string so that he wins. Here is an amusing fact: given Alice's string, Bob can use the optimal strategy in Penney's game to pick his winning string.

It is surprising that, given that Alice starts first, she still loses. This game has many intriguing properties, which we discuss in a separate paper \cite{NoFlippancy}.

Now, as Alice loses, we need to help her again. We decided to add probabilities to this game.

\section{The Blended game}\label{sec:c}

This game is a variation of the No-Flippancy game, where we added chances back. Alice and Bob choose strings of the same length. As in the No-Flippancy game, they each want to increase their prefix in the output by 1. If they want the same outcome of a coin toss, they get it. If they want different outcomes, they flip a coin.

As usual, we study when Alice and Bob have strings of length 3. 

If Alice's first two characters are the same as Bob's, then the probability of each of them winning is 1/2. Indeed, they want the same first two characters, after which they flip a coin ending a game. If they have complementary strings, by symmetry, the probability of each of them winning is 1/2.

We used a computer to calculate the results for other cases. Table~\ref{table:CG3Bobwinprob} shows Bob's winning probabilities.

\begin{table}[h!]
\begin{center}
\begin{tabular}{|c|c|c|c|c|}
\hline
3 & HHH & HHT & HTH & HTT \\
\hline
HHH & * & 1/2 & 1/4 & 2/5 \\
HHT & 1/2 & * & 1/2 & 2/3 \\
HTH & 3/4 & 1/2 & * & 1/2 \\
HTT & 3/5 & 1/3 & 1/2 & * \\
THH & 7/8 & 3/4 & 1/4 & 1/2 \\
THT & 7/12 & 3/8 & 1/2 & 3/4 \\
TTH & 7/10 & 1/2 & 5/8 & 1/4 \\
TTT & 1/2 & 3/10 & 5/12 & 1/8 \\
\hline
\end{tabular}
\end{center}
\caption{Bob's winning probabilities.}
\label{table:CG3Bobwinprob}
\end{table}

We see that a lot of entries in this table are the same as in Penney's game. However, some differ. Suppose Alice picks HHH and Bob picks HTH. The first toss is H as both want H. After that Alice wins only if the next flip is H, followed by another flip that is H.  

Bob's best choice and odds are in Table~\ref{table:CG3Bobodds}.

\begin{table}[h!]
\begin{center}
\begin{tabular}{ |c|c|c| } 
 \hline
 Alice's string & Bob's best choice & odds \\ 
 \hline
 HHH & THH & 7 to 1 \\ 
 \hline
 HHT & THH  & 3 to 1 \\ 
 \hline
 HTH & TTH & 5 to 3 \\
 \hline
 HTT & THT & 3 to 1 \\
 \hline
\end{tabular}
\end{center}
\caption{Bob's best choice and odds.}
\label{table:CG3Bobodds}
\end{table}

Figure~\ref{fig:Coop} shows Bob's best choice. The arrow points from Alice's string to Bob's best choice.

\begin{figure}[h!]
\centering
\includegraphics[scale=0.5]{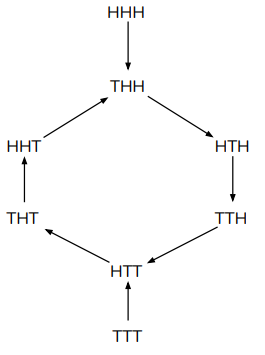}
\caption{Bob's best choice.}
\label{fig:Coop}
\end{figure}

The above graphic demonstrates something rather beautiful about this game: it has a non-transitive cycle of length 6. 
As in Penney's game, Bob always has better odds.

After playing these numerous coin games, Alice and Bob have had enough. They decide to give their exorbitantly lustrous coin, which, like all the others, is worth $10,000$ USD, to charity, and their discoveries to the world of coin-tastic mathematical research.

\section{Acknowledgements}

This project was done as part of MIT PRIMES STEP, a program that allows students in grades 6 through 9 to try research in mathematics. Tanya Khovanova is the mentor of this project. We are grateful to PRIMES STEP for this opportunity.

\end{document}